\documentclass[a4paper, 12pt, reqno]{amsart}

\usepackage{amsfonts}
\usepackage{amssymb}
\usepackage{latexsym}
\usepackage{graphics}
\usepackage[2emode]{psfrag}
\usepackage{mathrsfs}
\usepackage{amsthm}
\usepackage{amsmath}
\usepackage[all]{xy}
\usepackage{enumerate}
\usepackage[latin1]{inputenc}
\usepackage{lscape}
\usepackage{a4wide}
\usepackage[usenames]{color}
\usepackage{soul}

\usepackage[all]{pstricks}
\input{xy}
 \xyoption{all}

\newtheorem{theorem}{Theorem}[section]
\newtheorem{proposition}[theorem]{Proposition}
\newtheorem{corollary}[theorem]{Corollary}

\newtheorem*{theorem*}{Theorem}
\newtheorem*{proposition*}{Proposition}
\newtheorem*{corollary*}{Corollary}
\newtheorem*{lemma*}{Lemma}
\theoremstyle{definition}
\newtheorem{definition}[theorem]{Definition}
\newtheorem{punto}[theorem]{}
\newtheorem*{punto*}{ }

\newtheorem{remark}[theorem]{Remark}
\newtheorem*{remark*}{Remark}

\newtheorem*{definition*}{Definition}

\newcommand{\coring}[1]{\mathfrak{#1}}
\newcommand{\tensor}[1]{\otimes_{#1}}

\newcommand{\rmod}[1]{\mathsf{Mod}_{#1}}

\newcommand{\rcomod}[1]{\mathsf{Comod}_{#1}}

\newcommand{\cotensor}[1]{\square_{#1}}

\newcommand{\cat}[1]{\mathcal{#1}}

\newcommand{\upmatrix}[3]{\scriptscriptstyle{\begin{pmatrix}
  #1 & 0 \\
  #2 & #3
\end{pmatrix}}}

\begin{document}
\allowdisplaybreaks

\title{Hereditary triangular matrix comonads}
\author{Laiachi El Kaoutit}
\address{Universidad de Granada, Departamento de \'{A}lgebra. Facultad de Educaci\'on, Econom\'{\i}a y Tecnolog\'{\i}a. 
Cortadura del Valle s/n. E-51001 Ceuta, Spain}
\email{kaoutit@ugr.es}
\author{Jos\'e G\'omez-Torrecillas}
\address{Universidad de Granada.  Departamento de \'{A}lgebra, Facultad de Ciencias 
 E18071 Granada, Spain} 
\email{gomezj@ugr.es}
%\date{\today} 
\thanks{Research supported by the Spanish Ministerio de Econom\'{\i}a y Competitividad  and the European Union, grant MTM2013--41992-P}


\begin{abstract}We recognise Harada's generalized categories of diagrams as a particular case of modules over a monad defined on a finite direct product of additive categories. We work in the dual (albeit formally equivalent) situation, that is, with comodules over comonads.  
With this conceptual tool at hand, we obtain several of the Harada results with simpler proofs, some of them under more general hypothesis, besides with a characterization of the normal triangular matrix comonads that are hereditary, that is, of homological dimension less or equal than $1$. Our methods rest on a matrix representation of additive functors and natural transformations, which allows us to adapt typical algebraic manipulations from Linear Algebra to the additive categorical setting. 
\end{abstract}

\maketitle

\section*{Introduction}
Every complete set of pairwise orthogonal idempotent elements $e_1, \dots, e_n$ of a unital ring $R$ allows to express the ring as a generalized $n \times n$--matrix ring with entries in the bimodules $e_iRe_j$ for $1 \leq i, j \leq n$. The origin of this decomposition can be traced back to the seminal work of B. Peirce \cite{Peirce:1882}. Beyond the role of matrix rings in the classification of semi-simple artinian rings (i.e., Wedderburn-Artin's Theorem), these generalized matrix rings were used to investigate rings of low homological dimension, in the framework of  a systematic  program of  studying non commutative rings of finite homological dimension promoted by Eilenberg,  Ikeda,  Jans, Kaplansky, Nagao, Nakayama, Rosenberg and Zelinsky among others, see \cite{Eilenberg et all:1955, Eilenberg et all:1956, Eilenberg et all:1957, Jans/Nakayama:1957, Kaplansky:1958}.  In particular, S. U. Chase \cite{Chase:1961} and M. Harada  \cite{Harada:1966} used generalized triangular matrix rings to investigate the structure of the semi-primary hereditary rings (i.e., semi-primary rings with homological dimension $1$). 

Inspired by the study of homological properties of abelian categories of diagrams given in \cite[\S IX]{Mitchell:1965}, Harada formulated in \cite{Harada:1967} versions of some of his results on hereditary triangular matrix rings from \cite{Harada:1966} in the framework of the so called abelian categories of generalized commutative diagrams. In this paper, we recognize these categories as the categories of modules (or algebras) over suitable monads. This allows, apart from giving a more conceptual treatment of these categories, to obtain most of the main results from \cite{Harada:1967} with sharper (and simpler, we think) proofs. Our methods, based on a sort of ``Linear Algebra''  for additive functors and natural transformations, also provide some new results, including a characterization of hereditary categories of generalized commutative diagrams.  

We work in the dual, although  formally equivalent, situation than in \cite{Harada:1967}. Thus, we consider a comonad defined on a finite direct product of additive categories, and we investigate how to express in matrix form both its comultiplication and its counit.  In this way, its category of comodules (or coalgebras) is expressed in a such a way that Harada's categories of generalized diagrams become a particular case (namely, the categories of modules over a normal triangular monad). This is done in Section \ref{comonadmatrix}.

In Section \ref{comonadtriangular} we generalize \cite[Theorem 2.3]{Harada:1967} in several directions. On one hand, we do not  assume that the base categories are abelian and, on the other, our categories of comodules are more general than the dual of the categories considered in \cite{Harada:1967} (see Theorem \ref{un medio}). We derive our Theorem \ref{un medio} from a simpler result, namely Theorem \ref{thm:basico}, in conjunction with the kind of linear algebra for functors developed in Section \ref{comonadmatrix}. 

In Section \ref{sec:III} we give full proofs of  the dual form of some results from \cite{Harada:1967}. These results deal with the structure of the injective comodules over a hereditary normal triangular matrix comonad. In particular, they contain part of the dual of \cite[Theorem 3.6]{Harada:1967}. We also prove a characterization of hereditary normal triangular matrix comonads (see Theorem \ref{main-result}). 

Section \ref{sec:coalgebras} illustrates our general results by giving a characterization of the bipartite coalgebras (see \cite{Kosakowska/Simson:2008, Iovanov:2015}) that are right hereditary. Our methods allow us to work over a general commutative ring.

\section{Matrix comonads and their categories of comodules.}\label{comonadmatrix}

In this section we introduce the  notion of matrix comonad and we describe its category of comodules. We develop a kind of linear algebra for matrix  functors and their matrix natural transformations. 

\subsection{Basic notions and notation.}
All the categories in this paper are assumed to be additive, and all functors between them are additive.  If $A$ is an object of a category, then its identity morphism is also denoted by $A$. We shall use the standard notation for composition of functors and/or natural transformations, see \cite{Barr/Wells:1985}.

 If $A, A'$ are objects of an additive category $\cat{A}$, then $A \oplus A'$ denotes its direct sum, and similarly for any finite collection of objects. Recall that the direct sum of finitely many objects is both the product and the coproduct of the family in the category. The symbol $\oplus$ will be used also for the direct sum of morphisms. 

If $F, G : \cat{A} \to \cat{B}$ are functors, then its direct sum functor $F \oplus G : \cat{A} \to \cat{B}$ is given by the composition
\[
\xymatrix{F \oplus G : \cat{A} \ar^-{\Delta}[r] & \cat{A} \times \cat{A} \ar^-{F \times G}[r] & \cat{B} \times \cat{B} \ar^-{\oplus}[r] & \cat{B}},
\]
where $\Delta$ denotes the diagonal functor, and $\cat{C} \times \cat{D}$ denotes the product category of two categories $\cat{C}, \cat{D}$. Thus, $(F \oplus G)(f) = F(f) \oplus G(f)$ for any morphism $f$ of $\cat{A}$. The direct sum of finitely many functors is defined analogously.  

Let $\eta : F \to G \oplus H$ be a natural transformation, where $F, G, H : \cat{A} \to \cat{B}$ are functors. By $\pi_G : G \oplus H \to G$ (resp. $\pi_H : G \oplus H \to H$) we denote the natural transformation defined by the canonical projection. Therefore,  $\eta$ is uniquely determined by the natural transformations $\mu = \pi_G \eta : F \to G$ and $\nu = \pi_H \eta : F \to H$. We will then use the notation $\eta = \mu \dotplus \nu$. 

We will consider comonads (or cotriples) on $\cat{A}$. We refer to \cite{Barr/Wells:1985} for details on (co)monads and their categores of (co)modules (or (co)algebras).

\subsection{Matrix notation.}
Let $\cat{A}_1, \dots, \cat{A}_n, \cat{B}_1, \dots, \cat{B}_m$ be additive categories  and consider the product categories 
\begin{equation}\label{AB}
\cat{A} = \cat{A}_1 \times \cdots \times \cat{A}_n, \quad \cat{B} = \cat{B}_1 \times \cdots \times \cat{B}_m.
\end{equation}
 For every $j = 1, \dots, n$, let $\pi_j :
\cat{A} \rightarrow \cat{A}_j$ (resp.~$\iota_j : \cat{A}_j \rightarrow \cat{A}$) denote the canonical projection (resp. injection) functors, and similarly for $\cat{B}$. 
An additive functor $F : \cat{A} \rightarrow \cat{B}$ is determined by 
$m\times n$ functors $F_{ji} = \pi_i F \iota_j : \cat{A}_j \rightarrow
\cat{B}_i, j = 1, \dots, n, i = 1, \dots, m$, since, from the equalities
\[
1_{\cat{A}} = \iota_1\pi_1 \oplus \cdots \oplus \iota_n \pi_n, \quad 1_{\cat{B}} = \iota_1\pi_1 \oplus \cdots \oplus \iota_m \pi_m,\]
we obtain
\[
F = \bigoplus_{i,j} \iota_i\pi_iF\iota_j\pi_j =
\bigoplus_{i,j}\iota_i F_{ji}\pi_j.
\]
This means that given a  morphism  $f = (f_1, \dots,
f_n)$ in $\cat{A}$, we have 
\[
Ff = (\oplus_j F_{j1}f_j, \cdots, \oplus_j F_{jm}f_j).
\]
This expression can be represented in matrix form as 
\[
Ff = \left(%
\begin{array}{cccc}
  F_{11} & F_{21} & \cdots & F_{n1} \\
  F_{12} & F_{22} & \cdots & F_{n2} \\
  \vdots & \vdots &  & \vdots \\
  F_{1m} & F_{2m} & \cdots & F_{nm} \\
\end{array}%
\right)\left(%
\begin{array}{c}
f_1 \\
f_2 \\
\vdots \\
f_n
\end{array}%
\right)
\]

Now, if  $\eta : F \rightarrow G$ is a natural transformation, where $G : \cat{A} \to \cat{B}$ is a second functor,  then we have the natural transformations
\begin{equation}\label{Eq:etaij}
\eta^{ij} = \pi_j \eta {\iota_i} : F_{ij} \rightarrow G_{ij},  \qquad i = 1, \dots, n, \; j = 1, \dots, m 
\end{equation}
which completely determine $\eta$ as follows. For each $i = 1, \dots, n$, and each $A
\in \cat{A}$, we consider the canonical morphism $
\xi^i_A : \iota_i\pi_iA \rightarrow A$ given by the decomposition $
A = \iota_1\pi_1A \oplus \cdots \oplus \iota_n\pi_n A$.  From  the naturality of $\eta$, we get the following commutative diagrams
\[
\xymatrix{FA = F\iota_1\pi_1A \oplus \cdots \oplus F\iota_n\pi_nA
\ar^{\eta_A}[r] & GA = G\iota_1\pi_1A \oplus \cdots \oplus
G\iota_n\pi_nA \\
F\iota_i\pi_iA \ar_{\eta_{\iota_i\pi_iA}}[r] \ar^{F\xi^i_A}[u]
&G\iota_i\pi_iA \ar^{G\xi^i_A}[u] }
\]
which show  that $\eta_A = \oplus_i \eta_{\iota_i\pi_iA}$
and, thus, for each  $j = 1, \dots, n$, we have
\[
\pi_j\eta_A = \oplus_{i} \pi_j\eta_{\iota_i\pi_iA} = \oplus_{i}
\eta^{ij}_{\pi_iA}
\]
Since $\eta_A$ is defined by the morphisms  $\pi_j\eta_A$,  $j = 1,
\dots,  m$, we get that $\eta$ is entirely determined by $ \eta^{ij},
i = 1, \dots, n,  j= 1, \dots, m$. So we can represent our natural transformation in a matrix form
\[
\eta = \left(%
\begin{array}{cccc}
  \eta^{11} & \eta^{21} & \cdots & \eta^{n1} \\
  \eta^{12} & \eta^{22} & \cdots & \eta^{n2} \\
  \vdots & \vdots &  & \vdots \\
  \eta^{1 m} & \eta^{2  m} & \cdots & \eta^{n  m} \\
\end{array}%
\right) :
\left(%
\begin{array}{cccc}
  F_{11} & F_{21} & \cdots & F_{n1} \\
  F_{12} & F_{22} & \cdots & F_{n2} \\
  \vdots & \vdots &  & \vdots \\
  F_{1 m} & F_{2  m} & \cdots & F_{n  m} \\
\end{array}%
\right) \longrightarrow
\left(%
\begin{array}{cccc}
  G_{11} & G_{21} & \cdots & G_{n1} \\
  G_{12} & G_{22} & \cdots & G_{n2} \\
  \vdots & \vdots &  & \vdots \\
  G_{1  m} & G_{2  m} & \cdots & G_{n  m} \\
\end{array}%
\right).
\]

\subsection{Matrix operations.}
Let 
\[
\xymatrix{ \cat{A} \ar^-{F}[r] & \cat{B} \ar^-{G}[r] & \cat{C}}
\]
be functors, where $\cat{C} = \cat{C}_1 \times \cdots \times \cat{C}_l$, and $\cat{A}, \cat{B}$ are finite products as in \eqref{AB}.  Then the composite functor $H = GF$ is represented in matrix form as the ``usual matrix product'' of matrices representing $G$ and $F$. This means that
\[
H = \left( \begin{array}{cccc}
  H_{11} & H_{21} & \cdots & H_{n1} \\
  H_{12} & H_{22} & \cdots & H_{n2} \\
  \vdots & \vdots &  & \vdots \\
  H_{1 l} & H_{2  l} & \cdots & H_{n l} \\
\end{array}%\
\right), 
\]
where 
\[
H_{ji} = \bigoplus_{k=1}^{m} G_{ki}F_{jk}, \qquad  i = 1, \dots, l, \; j = 1, \dots, n.
\]

As for natural transformations concerns, let $\eta: F \rightarrow F', \mu: F' \rightarrow F'' $ be natural transformations, where $F, F', F'' : \cat{A} \to \cat{B}$ are functors. A straightforward argument shows that the matrix representing $\mu\eta$ is the ``Hadamard  product'' of the matrices that represent $\mu$ and $\eta$, namely
\begin{multline}
\mu \eta = \left( \begin{array}{cccc}
  \mu^{11} & \mu^{21} & \cdots & \mu^{n1} \\
  \mu^{12} & \mu^{22} & \cdots & \mu^{n2} \\
  \vdots & \vdots &  & \vdots \\
  \mu^{1 m} & \mu^{2  m} & \cdots & \mu^{n  m} \\
\end{array}\right) \cdot \left( \begin{array}{cccc}
  \eta^{11} & \eta^{21} & \cdots & \eta^{n1} \\
  \eta^{12} & \eta^{22} & \cdots & \eta^{n2} \\
  \vdots & \vdots &  & \vdots \\
  \eta^{1 m} & \eta^{2  m} & \cdots & \eta^{n  m} \\
\end{array}\right) \\ = \left( \begin{array}{cccc}
  \mu^{11}\eta^{11} & \mu^{21}\eta^{21} & \cdots & \mu^{n1}\eta^{n1} \\
  \mu^{12}\eta^{12} & \mu^{22}\eta^{22} & \cdots & \mu^{n2}\eta^{n2} \\
  \vdots & \vdots &  & \vdots \\
  \mu^{1 m}\eta^{1 m} & \mu^{2 m}\eta^{2  m} & \cdots & \mu^{n m}\eta^{n  m} \\
\end{array}\right) 
\end{multline}

Finally, let us consider the composition of a functor and a natural transformation. So let $\eta : F \rightarrow F'$ with $F, F' : \cat{A} \to \cat{B}$, and $G : \cat{B} \to \cat{C}$. The natural transformation $G \eta : GF \rightarrow GF'$ is easily shown to have a matrix representation 
\[
G \eta = \left( \begin{array}{cccc}
  (G\eta)^{11} & (G\eta)^{21} & \cdots & (G\eta)^{n1} \\
  (G\eta)^{12} & (G\eta)^{22} & \cdots & (G\eta)^{n2} \\
  \vdots & \vdots &  & \vdots \\
  (G\eta)^{1 l} & (G \eta)^{2  l} & \cdots & (G \eta)^{n l} \\
\end{array}%\
\right), 
\]
where 
\[
(G \eta)^{ji} = \bigoplus_{k=1}^m G_{ki}\eta^{jk}, \qquad  i = 1, \dots, l, \; j = 1, \dots, n.
\]
That is, the composition $G \eta$ leads to the ``usual matrix product''. This also holds for the composition in the opposite order, that is, given a natural transformation $\gamma : G \rightarrow G'$, for functors $G, G' : \cat{B} \to \cat{C}$, and  $F : \cat{A} \rightarrow \cat{B}$, the matrix representation of the natural transformation $\gamma F : GF \rightarrow G'F$ is given by
\[
\gamma F = \left( \begin{array}{cccc}
  (\gamma F)^{11} & (\gamma F)^{21} & \cdots & (\gamma F)^{n1} \\
  (\gamma F)^{12} & (\gamma F)^{22} & \cdots & (\gamma F)^{n2} \\
  \vdots & \vdots &  & \vdots \\
  (\gamma F)^{1 l} & (\gamma F)^{2  l} & \cdots & (\gamma F)^{n l} \\
\end{array}%\
\right), 
\]
where
\[
(\gamma F)_{ji} = \bigoplus_{k=1}^m \gamma^{ki}F_{jk}, \qquad  i = 1, \dots, l, \; j = 1, \dots, n.
\]

\subsection{Matrix comonads.}
Fix a category $\cat{A} = \cat{A}_1 \times \cdots \times \cat{A}_n$. Let  $F : \cat{A} \to \cat{A}$ be an endofunctor, and 
$\delta : F \to F^2$ a natural transformation.  In matrix form, 
\[
\delta = (\delta^{ij}) : (F_{ij}) \to (F_{ij})(F_{ij}), 
\]
where $\delta^{ij} : F_{ij} \rightarrow \oplus_k F_{kj}F_{ik}$ is a natural transformation for every $i,j = 1, \dots, n$, which is uniquely expressed as
 \[
\delta^{ij} = \delta^{i1j} \dotplus \delta^{i2j} \dotplus \cdots \dotplus \delta^{inj}
\]
for some natural transformations  $\delta^{ikj} : F_{ij} \rightarrow
F_{kj}F_{ik}$ with  $k = 1, \dots, n$. The coassociativity of $F$ is given by the equation 
$
(\delta F) \delta = (F\delta) \delta
$
or, equivalently, 
\begin{equation}\label{coassij}
(\delta F)^{ij} \delta^{ij} = (F\delta)^{ij} \delta^{ij}, \quad 1
\leq i,j \leq n.
\end{equation}

We know that 
\[
(\delta F)^{ij} = \bigoplus_{k}\delta^{kj}{F_{ik}}
\quad \text{ and } \,\,\,\,  
(F\delta)^{ij} = \bigoplus_{l}F_{lj}\delta^{il}.
\]
Observe that for each $i,j,k = 1, \dots, n$, the natural transformation 
$\delta^{kj}F_{ik} : F_{kj}F_{ik} \rightarrow
\bigoplus_{l}F_{lj}F_{kl}F_{ik}$ is given by 
\[
\delta^{kj}F_{ik} = \delta^{k1j} F_{ik} \dotplus  \delta^{k2j} F_{ik} \dotplus \cdots \dotplus
\delta^{knj} F_{ik}
\]
while, for each $i,j,l = 1, \dots, n$, the natural transformation 
$F_{lj}\delta^{il} : F_{lj}F_{il} \rightarrow
\bigoplus_{k}F_{lj}F_{kl}F_{ik}$ is defined by 
\[
F_{lj}\delta^{il} = F_{lj}\delta^{i1l} \dotplus 
F_{lj}\delta^{i2l} \dotplus \cdots \dotplus
F_{lj}\delta^{inl}.
\]

We then obtain that 
\[
(\delta F)^{ij}\delta^{ij} = \left(\bigoplus_k
\delta^{kj} {F_{ik}}\right)\left(\underset{k}\dotplus \delta^{ikj}\right) =
\underset{k,l}\dotplus ( \delta^{klj} {F_{ik}} )\delta^{ikj} 
\]
and
\[
(F\delta)^{ij} \delta^{ij} = \left(
\bigoplus_{l}F_{lj}\delta^{il}\right)\left(\underset{l}\dotplus \delta^{ilj}\right) =  \underset{k,l}\dotplus (F_{lj}\delta^{ikl})\delta^{ilj} .
\]
Both vectors should be equal, which is equivalent to
\[
(\delta^{klj} F_{ik}) \delta^{ikj} =
(F_{lj}\delta^{ikl})\delta^{ilj}, \quad (1 \leq i,j,k,l \leq n).
\]
That is, all the diagrams 
\[
\xymatrix{F_{ij} \ar^{\delta^{ilj}}[r] \ar_{\delta^{ikj}}[d] & F_{lj}F_{il} \ar^{F_{lj}\delta^{ikl}}[d]\\
F_{kj}F_{ik} \ar_{\delta^{klj}{F_{ik}}}[r] & F_{lj}F_{kl}F_{ik}}
\]
are commutative.

Next, we will discuss when a comultiplication $\delta : F \to F^2$ is counital, for a given counity $\varepsilon : F \to 1_{\cat{A}}$. The general matrix form of $\varepsilon$ is

\begin{equation}\label{Eq:epsilon}
\varepsilon = \left( %
\begin{array}{cccc}
\varepsilon^1 & 0 & \cdots & 0 \\
0 & \varepsilon^2 & \cdots & 0 \\
\vdots & \vdots & & \vdots \\
0 & 0 & \cdots & \varepsilon^n
\end{array}
\right) : %
\left(%
\begin{array}{cccc}
F_{11} & F_{21} & \cdots & F_{n1} \\
F_{12} & F_{22} & \cdots & F_{n2} \\
\vdots & \vdots & & \vdots \\
F_{1n} & F_{2n} & \cdots & F_{nn}
\end{array}
\right) \longrightarrow %
\left( %
\begin{array}{cccc}
1 & 0 & \cdots & 0 \\
0 & 1 & \cdots & 0 \\
\vdots & \vdots & & \vdots \\
0 & 0 & \cdots & 1
\end{array}
\right)
\end{equation}
for some natural transformations $\varepsilon^i : F_{ii} \to 1_{\cat{A}_i}$ for $i = 1, \dots, n$. The counitality conditions $(\varepsilon F) \delta = 1_F = (F \varepsilon) \delta$ lead to matrix equalities equivalent to
\[
(\varepsilon^j F_{ij})\delta^{ijj} = 1_{F_{ij}} = (F_{ij}\varepsilon^i)\delta^{iij}, \qquad (i,j = 1 \dots,n) 
\]
that is, all the diagrams
\[
\xymatrix{F_{ij} \ar@{=}[rd] \ar^-{\delta^{ijj}}[r] \ar_-{\delta^{iij}}[d] & F_{jj}F_{ij} \ar^-{\varepsilon^jF_{ij}}[d] \\
F_{ij}F_{ii} \ar_-{F_{ij}\varepsilon^i}[r] & F_{ij}}
\]
commute. 

We have so far proved the following proposition. 

\begin{proposition}\label{comonadasmatrices}
Let $\cat{A} = \cat{A}_1 \times \dots \times \cat{A}_n$ be the product category of finitely many categories $\cat{A}_1, \dots, \cat{A}_n$, and $F : \cat{A} \to \cat{A}$ any functor.  There is a bijective correspondence between
\begin{enumerate}
\item Comonads $(F, \delta, \varepsilon)$;
\item Sets of natural transformations  $\{ \delta^{ikj} :
F_{ij} \rightarrow F_{kj}F_{ik} : i,j, k = 1, \dots, n \}$, and $\varepsilon^i : F_{ii} \to 1_{\cat{A}_i} :  i = 1, \dots, n$ such that 
\begin{enumerate}
\item For all $i, j, k, l = 1, \dots, n$ the diagram  
\begin{equation}\label{Eq:cuadrados}
\xymatrix{F_{ij} \ar^{\delta^{ilj}}[r] \ar_{\delta^{ikj}}[d] & F_{lj}F_{il} \ar^{F_{lj}\delta^{ikl}}[d]\\
F_{kj}F_{ik} \ar_{\delta^{klj}{F_{ik}}}[r] & F_{lj}F_{kl}F_{ik}} 
\end{equation}
conmutes, and %
\item for all $i,j = 1, \dots, n$, the diagram
\begin{equation}\label{Eq:triangulos}
\xymatrix{F_{ij} \ar@{=}[rd] \ar^-{\delta^{ijj}}[r] \ar_-{\delta^{iij}}[d] & F_{jj}F_{ij} \ar^-{\varepsilon^jF_{ij}}[d] \\
F_{ij}F_{ii} \ar_-{F_{ij}\varepsilon^i}[r] & F_{ij}}
\end{equation}
commutes. 
\end{enumerate}
\end{enumerate}
\end{proposition}

\begin{remark}\label{remark:CT}
Take $i=j=k$ in diagrams \eqref{Eq:cuadrados} and \eqref{Eq:triangulos}, then $(F_{ii}, \delta^{iii},\varepsilon^i)$ is a comomad over $\cat{A}_{i}$.  Also diagrams \eqref{Eq:cuadrados} and \eqref{Eq:triangulos} show that each functor $F_{ij}$ is in fact an $F_{jj}-F_{ii}$-bicomodule functor, in the sense of \cite[Definition 4.7]{Bruni:2011}, see also \cite{Kaoutit/Vercruysse:2009}. Furthermore, if $l=k$ in diagram \eqref{Eq:cuadrados}, we get that each of the $F_{ik}$'s is a balanced bicomodule in a dual sense of \cite[\S 3.2]{Bruni:2011}. In this way, if the 'cotensor product'  functor $F_{kj} \square_{F_{kk}}F_{ik}$ do exist, then the natural transformation $\delta^{ikj}$ factors through   $F_{kj} \square_{F_{kk}}F_{ik}$.
\end{remark}

\begin{definition}
A comonad $(F,\delta,\varepsilon)$ on $\cat{A} = \cat{A}_1 \times \dots \times \cat{A}_n$ is called \emph{normal} if $F_{ii} = 1_{\cat{A}_i}$ and $\varepsilon^i = 1$ for all $i = 1, \dots, n$. Thus, for a normal comonad, we have necessarily that $\delta^{ijj} = \delta^{iij} = 1_{F_{ij}}$ for all $i,j = 1, \dots, n$.
\end{definition}

\begin{remark}
In the case $n = 2$, we deduce from Proposition \ref{comonadasmatrices} that a normal comonad $(F,\delta,\varepsilon)$ is given by natural transformations 
$\delta^{121} : 1_{\cat{A}_1} \rightarrow F_{21}F_{12}$ and
$\delta^{212} : 1_{\cat{A}_2} \rightarrow F_{12}F_{21}$ such that 
$F_{12}\delta^{121} = \delta^{212}{F_{12}}$ and 
$F_{21}\delta^{212} = \delta^{121}{F_{21}}$. Therefore, the normal comonads over  $\cat{A}_1 \times \cat{A}_2$ are in bijection with the wide (right) Morita contexts between  $\cat{A}_1$ and $\cat{A}_2$ as defined in \cite{Castano/Gomez:1995, Castano/Gomez:1998}. 
\end{remark}

\begin{remark}
It is possible to formulate Proposition \ref{comonadasmatrices} in dual form, thus given the structure of the monads on $\cat{A}_1 \times \cdots \times \cat{A}_n$. 
\end{remark}

\subsection{Comodules over matrix comonads}

Consider a comonad $(F,\delta,\varepsilon)$ over $\cat{A} = \cat{A}_1 \times \cdots \times \cat{A}_n$ as in  Proposition  \ref{comonadasmatrices}. Next we want to describe its Eilenberg-Moore category  $\cat{A}_F$ of comodules (or coalgebras). Recall that an object of $\cat{A}_F$ is a morphism $\mathbf{d}_A : A \rightarrow
FA$ in $\cat{A}$ such that the following diagrams are commutative:
\begin{equation}\label{coalgebras}
\xymatrix{A \ar^{\mathbf{d}_A}[r] \ar_{\mathbf{d}_A}[d]& FA \ar^{\delta_A}[d] \\
FA \ar_{F\mathbf{d}_A}[r] & F^2A} \qquad \xymatrix{A
\ar^{\mathbf{d}_A}[r] \ar_{1_A}[dr] & FA \ar^{\varepsilon_A}[d] \\
& A}
\end{equation}
We want to describe these objects in terms of the categories 
$\cat{A}_1, \dots, \cat{A}_n$. Each structure morphism 
$\mathbf{d}_A$ is determined by the projections
\[
\pi_j(\mathbf{d}_A) : A_j \rightarrow \bigoplus_{k}F_{kj}A_k,
\qquad j = 1, \dots, n
\]
and each of them is given by 
\[
\pi_j (\mathbf{d}_A) = 
d_A^{1j} \dotplus
d_A^{2j} \dotplus
\cdots \dotplus
d_A^{nj}
 \quad \hbox{ with  } \,\,\, d_A^{kj} : A_j \rightarrow F_{kj}A_k.
\]

Let us see what restrictions impose the commutativity of the diagrams \eqref{coalgebras} on the morphism  $d_A^{kj}$. We start with the equation  $\delta_A \mathbf{d}_A =
F(\mathbf{d}_A)\mathbf{d}_A$, which lead to the following system of equations
\begin{equation}\label{comodecu}
\pi_j(\delta_A)\pi_j(\mathbf{d}_A) =
\pi_j(F(\mathbf{d}_A))\pi_j(\mathbf{d}_A), \qquad j = 1, \dots, n.
\end{equation}
On the other hand,
\[
\pi_j\delta_A = \bigoplus_l\delta_{A_l}^{lj}, \hbox{ with  }\,\,
\delta_{A_l}^{lj} : F_{lj} A_l\rightarrow
\bigoplus_{k}F_{kj}F_{lk}A_l.
\]
Therefore, 
\begin{equation}\label{deltad}
\pi_j(\delta_A) \pi_j(\mathbf{d}_A) = \left( \bigoplus_l
\delta_{A_l}^{lj}\right) \left( \underset{l}{\dotplus} \, d_A^{lj} \right) = \underset{l}{\dotplus} \, \delta^{lj}_{A_l} d_A^{lj} = \underset{l}{\dotplus} \, \left( \underset{k}{\dotplus}\, \delta^{lkj}_{A_l}\right) d_A^{lj} = \underset{k,l}{\dotplus}\, \delta^{lkj}_{A_l} d_A^{lj}.
\end{equation}

On the other hand,
\begin{multline}\label{Fdd}
\pi_j(F\mathbf{d}_A)\pi_j(\mathbf{d}_A)  = \left(
\bigoplus_kF_{kj}\pi_k(\mathbf{d}_A) \right)\left( \underset{k}{\dotplus}\, d_A^{kj}\right) = \underset{k}{\dotplus}\, F_{kj} \pi_k (\mathbf{d}_A)d_A^{kj} = \\  \underset{k}{\dotplus}\, \left( \underset{l}{\dotplus}\, F_{kj}(d_A^{lk})\right)d_A^{kj} = \underset{k,l}{\dotplus}\, F_{kj}(d_A^{lk})d_A^{kj}.
\end{multline}

We get from \eqref{comodecu}, \eqref{deltad} and  \eqref{Fdd}  that the commutativity of the first diagram in  \eqref{coalgebras} is equivalent to the commutativity of all the diagrams
\begin{equation}\label{deltad=Fdd}
\xymatrix{A_j \ar^{d_A^{kj}}[r] \ar_{d_A^{lj}}[d] & F_{kj}A_k
\ar^{F_{kj}d_A^{lk}}[d] \\
F_{lj}A_l \ar_{\delta_{A_l}^{lkj}}[r] & F_{kj}F_{lk}A_l}
\end{equation}
for $j,k,l = 1, \dots, n$.

The second diagram in \eqref{coalgebras} leads to the equalities 
$\pi_j(\varepsilon_A \mathbf{d}_A) = 1_{A_j}$ for each $j=1,
\dots, n$. Therefore, 
\[
1_{A_j} = \pi_j(\varepsilon_A \mathbf{d}_A) =
\pi_j(\varepsilon_A)\pi_j(\mathbf{d}_A) =
 (0 \oplus \cdots \oplus 0 \oplus {\varepsilon^j_{A_j}} \oplus 0 \oplus \cdots \oplus 0)\left(\underset{k}{\dotplus}\, d_A^{kj}\right) = \varepsilon_{A_j}^j d_A^{jj},
\]
for all $j = 1, \dots, n$. 

The previous discussion gives the following description of the category $\cat{A}_F$ of comodules over $F$.

\begin{proposition}\label{catcomodulos}
Let $\cat{A} = \cat{A}_1 \times \cdots \times \cat{A}_n$ be the product category of finitely many categories $\cat{A}_1, \dots, \cat{A}_n$. If $(F,\delta,\varepsilon)$ is a comonad over $\cat{A}$, then the category of $F$--comodules is described as follows:
\begin{enumerate}
\item \underline{Objects:} They are pairs $(A,\mathbf{d}_A)$ where $A =
(A_1, \dots, A_n) \in \cat{A} = \cat{A}_1 \times \cdots \times
\cat{A}_n$ and $\mathbf{d}_A$ is a set of morphisms 
$\mathbf{d}_A = \{ d_A^{kj} : A_j \rightarrow F_{kj}A_k : 1 \leq j, k \leq n
\}$  such that the diagrams 
\begin{equation}\label{Eq:coaction}
\xymatrix{A_j \ar^{d_A^{kj}}[r] \ar_{d_A^{lj}}[d] & F_{kj}A_k
\ar^{F_{kj}d_A^{lk}}[d] \\
F_{lj}A_l \ar_{\delta_{A_l}^{lkj}}[r] & F_{kj}F_{lk}A_l} \qquad
\xymatrix{A_j \ar^-{d_A^{jj}}[r] \ar@{=}[dr] & F_{jj}A_j \ar^-{\varepsilon^j_{A_j}}[d] \\ & A_j }
\end{equation}
are commutative for  all $j,k,l = 1, \dots, n$. %
\item \underline{Morphisms:} A morphism $f : (A,\mathbf{d}_A)
\rightarrow (B,\mathbf{d}_{B})$ is a set of morphisms  $f =
\{ f_{j} : A_{j} \rightarrow B_j : j = 1, \dots, n \}$ such that the diagrams 
\begin{equation}\label{Eq:morph}
\xymatrix{A_j \ar^{d_A^{kj}}[r] \ar_{f_j}[d] & F_{kj}A_k
\ar^{F_{kj}f_k}[d] \\
B_j \ar_{d_B^{kj}}[r] & F_{kj}B_k }
\end{equation}
commute, for all $j, k = 1, \dots, n$.
\end{enumerate}
\end{proposition}

\begin{remark}\label{monad}
Obviously, it is possible to give the dual statements of the above results in the case of monads.   So a monad $T$ over 
$\cat{A} = \cat{A}_1 \times \cdots \times \cat{A}_n$
is given by a set of functors  $\{ T_{ij} : \cat{A}_i
\rightarrow \cat{A}_j, i,j = 1, \dots, n \}$ with two sets of natural transformations $\{ \mu^{ikj} : T_{kj}T_{ik} \rightarrow T_{ij} : i,j,k
= 1\dots,n \}$, $ \{ \eta^i : 1_{\cat{A}_i} \to T_{ii}, i = 1, \dots, n \}$ such that the diagrams
\[
\xymatrix@C=40pt{T_{lj}T_{kl}T_{ik} \ar^{\mu^{klj}{T_{ik}}}[r]
\ar_{T_{lj}\mu^{ikl}}[d] & T_{kj}T_{ik} \ar^{\mu^{ikj}}[d] \\
T_{lj}T_{il} \ar^{\mu^{ilj}}[r] & T_{ij}}, \qquad \xymatrix{T_{ij} \ar@{=}[rd] \ar^-{\eta^jT_{ij}}[r] \ar_-{T_{ij}\eta^i}[d] & T_{jj}T_{ij} \ar^-{\mu^{ijj}}[d] \\
T_{ij}T_{ii} \ar_-{\mu^{iij}}[r] & T_{ij}}\]
commute for all $i,j,k \in \{1, \dots, n\}$. 
The corresponding  category of modules   (or algebras) is described by dualizing the statements of Proposition \ref{catcomodulos}. When $T$ is a normal matrix monad (i.e., $T_{ii} = 1_{\cat{A}_i}$ and $\eta^i = 1$ for all $i = 1, \dots, n$), $\cat{A}$ is abelian, and $T_{ij} = 0$ for $i >j$, these are the kind of categories which were studied by M. Harada in \cite{Harada:1967}, and refereed to as \emph{categories of generalized diagrams} in abelian categories (see also \cite{Mitchell:1965}).
\end{remark}

\section{Triangular matrix comonads.}\label{comonadtriangular}

We have seen in Remark \ref{monad} that the abelian categories investigated by Harada in \cite{Harada:1967} are categories of modules over triangular normal matrix monads. In this section, we will take advantage of the fact that the existence of a (co)monad representing these categories to give a more systematic approach to their study. In fact, our main theorem in this section is more general in several directions than \cite[Theorem 2.3]{Harada:1967}. At the same time, we find the proof presented here sharpest in some aspects than Harada's one. 

We will say (see Definition \ref{def:comtriangular}) that a matrix comonad $F = (F_{ij})$ is \emph{triangular} if $F_{ij}= 0$ for all $1\leq j < i \leq n$.  The understanding of the case $n = 2$ is the key to the study of the general case.   

\subsection{Triangular matrix comonads of order $2$}\label{3comonad}
Let 
\begin{equation}\label{RVT}
\xymatrix{\cat{C} \ar^-{R}[r] & \cat{C} \ar^-{V}[r] & \cat{D} \ar^-{S}[r] & \cat{D}}
\end{equation}
be functors. By Proposition \ref{comonadasmatrices}, every comonad structure over the endofunctor 
\begin{equation}\label{triangular}
F = \left( \begin{array}{cc}
R & 0 \\
V & S
\end{array}\right)
\end{equation}
of $\cat{C} \times \cat{D}$
 is given by a set of natural transformations
\begin{equation}\label{sextupla}
\left\{ \xymatrix{R \ar^-{\delta^R}[r] & R^2, \; S \ar^-{\delta^S}[r] & S^2, \; V \ar^-{\rho^V} [r] & VR, \; V \ar^-{\lambda^V}[r] & SV, \; R \ar^-{\varepsilon^R}[r] & 1_{\cat{C}}, \; S \ar^-{\varepsilon^S}[r] & 1_{\cat{D}} }\right\}
\end{equation}
such that the following eight diagrams commute.
\begin{equation}\label{soncomonadas}
\xymatrix{R \ar^-{\delta^R}[r] \ar_-{\delta^R}[d] & R^2 \ar^-{R \delta^R}[d] \\
R^2 \ar_-{\delta^R R}[r] & R^3 }\quad
\xymatrix{S \ar^-{\delta^S}[r] \ar_-{\delta^S}[d] & S^2 \ar^-{S \delta^S}[d] \\
S^2 \ar_-{\delta^S S}[r] & S^3 }\quad
\xymatrix{R \ar@{=}[rd] \ar^-{\delta^R}[r] \ar_-{\delta^R}[d] & R^2 \ar^-{\varepsilon^R R}[d] \\
R^2 \ar_-{R\varepsilon^R}[r] & R}\quad
\xymatrix{S \ar@{=}[rd] \ar^-{\delta^S}[r] \ar_-{\delta^S}[d] & S^2 \ar^-{\varepsilon^S S}[d] \\
S^2 \ar_-{S\varepsilon^S}[r] & S}
\end{equation}

\begin{equation}\label{esbicomodulo}
\xymatrix{V \ar^-{\rho^V}[r] \ar_-{\rho^V}[d] & VR \ar^-{V \delta^R}[d] \\
VR \ar_-{\rho^V R}[r] & V R^2 }\quad 
\xymatrix{V \ar^-{\lambda^V}[r] \ar_-{\lambda^V}[d] & SV \ar^-{\delta^S V}[d] \\
SV \ar_-{S \lambda^V }[r] & S^2 V }\quad
\xymatrix{V \ar^-{\lambda^V}[r] \ar@{=}[dr] \ar_-{\rho^V}[d]& SV \ar^-{\varepsilon^S V}[d] \\
VR \ar_-{V \varepsilon^R}[r] & V }\quad
\xymatrix{V \ar^-{\lambda^V}[r] \ar_-{\rho^V}[d]  & SV  \ar^-{S \rho^V}[d]\\
VR \ar_-{\lambda^V R}[r] & SVR}
\end{equation}

By Remark \ref{remark:CT}, the commutativity of the diagrams \eqref{soncomonadas} just says that $(R,\delta^R,\varepsilon^R)$ and $(S,\delta^S,\varepsilon^S)$ are comonads, while the commutative  diagrams \eqref{esbicomodulo} say that $(V,\lambda^V,\rho^V)$ is an \emph{$S-R$--bicomodule functor}.

\begin{theorem}\label{thm:basico}
Let 
\[
\xymatrix{{\upmatrix{R}{V}{S}}: \cat{C} \times \cat{D} \ar[rr] &  & \cat{C} \times \cat{D}}
\] 
be a triangular matrix comonad of order $2$, with the comonad structure given by a sextuple of natural transformations as in \eqref{sextupla} satisfying the conditions \eqref{soncomonadas} and \eqref{esbicomodulo}. If the equalizer of every pair of arrows of the form $\xymatrix{VC \ar@<0.5ex>^-{\rho^V_C}[r] \ar@<-0.5ex>_-{V d_C}[r] & VRC }$, where $(C, d_C)$ is any $R$--comodule, do exist in $\cat{D}$, and $S$ preserves all these equalizers, then there exists a functor $T : \cat{C}_R \to \cat{D}_S$ such that the following diagram commutes,
\[
\xymatrix{\cat{C} \ar^-{V}[r] \ar_-{L}[d]& \cat{D} \\
\cat{C}_R \ar^-{T}[r] & \cat{D}_S \ar_-{U}[u]},
\]
where $L : \cat{C} \to \cat{C}_R$ is the free functor and $U : \cat{D}_S \to \cat{D}$ is the forgetful functor. 
Moreover, there exists an equivalence of categories
\[
(\cat{C} \times \cat{D})_{\scriptscriptstyle{\upmatrix{R}{V}{S}}} \cong (\cat{C}_R \times \cat{D}_S)_{\upmatrix{1}{T}{1}}
\]
\end{theorem}
\begin{proof}
The existence of $T$ is proved in \cite[Proposition 4.29]{Bruni:2011} in a different context. We give a direct construction for the convenience of the reader. Given an $R$--comodule  $(C,d_C)$, define an object $TC$ of $\cat{D}$ as the equalizer
\begin{equation}\label{eq:T}
\xymatrix{T
C \ar^-{\iota_C}[r] & VC \ar@<0.5ex>^-{\rho^V_C}[r] \ar@<-0.5ex>_-{V d_C}[r] & VRC }
\end{equation}
The $S$--coaction $\lambda^T_C: TC \to STC $, making $TC$ an $S$--comodule, is given by the universal property of the equalizer at the bottom of the following diagram.
\[
\xymatrix{T
C \ar^-{\iota_C}[r] \ar_-{\lambda^T_C}[d] & VC \ar_-{\lambda^V_C}[d]\ar@<0.5ex>^-{\rho^V_C}[r] \ar@<-0.5ex>_-{V d_C}[r] & VRC \ar^-{\lambda^V_{RC}}[d]\\
ST
SC \ar^-{S\iota_C}[r] & SVC \ar@<0.5ex>^-{S\rho^V_C}[r] \ar@<-0.5ex>_-{SV d_C}[r] & SVRC}
\]
Some straightforward computations show that this gives the object part of a functor $T : \cat{C}_R \to \cat{D}_S$. We only make explicit here its definition on morphisms. Given a morphism of $R$--comodules $f : (C,d_C) \to (C',d_{C'})$, the morphism $Tf : TC \to TC'$ is uniquely determined by the universal property of the equalizer in the bottom row of the serially commutative diagram
\[
\xymatrix{T
C \ar^-{\iota_C}[r] \ar_-{Tf}[d] & VC \ar_-{Vf}[d]\ar@<0.5ex>^-{\rho^V_C}[r] \ar@<-0.5ex>_-{V d_C}[r] & VRC \ar^-{VRf}[d]\\
TC' \ar^-{\iota_{C'}}[r] & VC' \ar@<0.5ex>^-{\rho^V_{C'}}[r] \ar@<-0.5ex>_-{V d_{C'}}[r] & VRC'}
\]

In order to construct an equivalence $E : (\cat{C} \times \cat{D})_{\upmatrix{R}{V}{S}} \to (\cat{C}_R \times \cat{D}_S)_{\upmatrix{1}{T}{1}}$, let us describe the objects of these categories of comodules. By Proposition \ref{catcomodulos}, a comodule over $\upmatrix{R}{V}{S}$ consists of an object $A = (C,D) \in \cat{C} \times \cat{D}$ and a tern of morphisms 
\[
\mathbf{d}_A = \{ \xymatrix{C \ar^-{\rho_C}[r] & RC}, \xymatrix{D \ar^-{d}[r] & VC}, \xymatrix{D \ar^-{\rho_D}[r] & SD} \}
\]
such that the following diagrams commute.
\begin{equation}\label{comodRVS1}
\xymatrix{C \ar^-{\rho_C}[r] \ar_-{\rho_C}[d] & RC \ar^-{R\rho_C}[d]\\
RC \ar_-{\delta^R_C}[r] & R^2C}\quad
\xymatrix{D \ar^-{\rho_D}[r] \ar_-{\rho_D}[d] & SD \ar^-{S\rho_D}[d]\\
SD \ar_-{\delta^S_D}[r] & S^2C}\quad
\xymatrix{D \ar^-{d}[r] \ar_-{d}[d] & VC \ar^-{V\rho_C}[d]\\
VC \ar_-{\rho^V_C}[r] & VRC}\quad
\xymatrix{D \ar^-{\rho_D}[r] \ar_-{d}[d] & SD \ar^-{Sd}[d]\\
VC \ar_-{\lambda^V_C}[r] & SVC}
\end{equation}
\begin{equation}\label{comodRVS2}
\xymatrix{C \ar^-{\rho_C}[r] \ar@{=}[dr] & RC \ar^-{\varepsilon^R_C}[d]\\
& C}\quad
\xymatrix{D \ar^-{\rho_D}[r] \ar@{=}[dr] & SD \ar^-{\varepsilon^S_D}[d]\\
& D}
\end{equation}
On the other hand, a comodule over $\upmatrix{1}{T}{1}$ is just a pair of comodules $((C,\rho_C),(D,\rho_D)) \in \cat{C}_R \times \cat{D}_S$ connected with a morphism of $S$--comodules $d' : D \to TC$. This last condition is just the commutativity of the diagram
\begin{equation}\label{comod1T1}
\xymatrix{D \ar^-{d'}[r] \ar_-{\rho_D}[d]& TC \ar^-{\lambda^T_C}[d]\\
SD \ar_-{Sd'}[r] & STC}
\end{equation}
The functor $E$ will send a tern $\mathbf{d}_A = (\rho_C,d,\rho_D)$ satisfying the conditions \eqref{comodRVS1} and \eqref{comodRVS2} to the pair of comodules $(C,\rho_C), (D,\rho_D) \in \cat{C}_R \times \cat{D}_S$ with the morphism $d':D \to TC$ is given by the universal property of the equalizer in the following diagram
\[
\xymatrix{TC \ar^-{\iota_C}[r] & VC \ar@<0.5ex>^-{\rho_C^V}[r] \ar@<-0.5ex>_-{Vd_C}[r] & VRC \\
D \ar^-{d'}[u] \ar_-{d}[ru] & & }
\]
The existence and uniqueness of $d'$ is guaranteed by the commutativity of the third diagram of \eqref{comodRVS1}, while the fact that $d'$ becomes a morphism of $S$--comodules (namely, the commutativity of \eqref{comod1T1}) is given by the fourth diagram in \eqref{comodRVS1}. This gives the object part of the functor $E$. On the other direction, there is a functor $E'$ defined on objects by sending a morphism of comodules $d' : D \to TC$, for $(C,\rho_C) \in \cat{C}_R$, $(D,\rho_D) \in \cat{D}_S$, to the morphism $d = \iota_C d'$. The computation
\[
\rho_C^V d = \rho_C^V\iota_Cd' = V \rho_C \iota_C d' = V \rho_C d
\]
shows that $d$ makes commute the third diagram in \eqref{comodRVS1}, while the computation
\[
\lambda_C^Vd = \lambda_C^V\iota_Cd' = S \iota_C \lambda^T_C d' = S \iota_C S d' \rho_D = (Sd)\rho_D
\]
is just the commutativity of the last diagram of \eqref{comodRVS1}. 

It is not hard to see that $E$ and $E'$ are mutually inverse. 
\end{proof}
\begin{remark}\label{rem:leftexact}
A standard argument shows that, if $\cat{C}$ and $\cat{D}$ are abelian categories, and $V$, $R$ and $S$ are left exact functors, then $T$ is a left exact functor between abelian categories.  
\end{remark}

\subsection{Triangular matrix comonads.}
\begin{definition}\label{def:comtriangular}
A functor $F: \cat{A}_1 \times \cdots \times \cat{A}_n \to \cat{A}_1 \times \cdots \times \cat{A}_n$ which is endowed with a comonad structure $(F,\delta,\varepsilon)$ as in Proposition \ref{comonadasmatrices}  is called a \emph{triangular matrix comonad of order $n$} over  $\cat{A} = \cat{A}_1 \times \cdots \times \cat{A}_n$ if $F_{ji} = 0$ for all $j >i$, that is, 
\begin{equation}\label{ntriangular}
F = \left( \begin{array}{cccccc} F_{11} & 0 & 0 &\cdots & 0&0 \\
F_{12} & F_{22} & 0 & \cdots & 0& 0 \\
\vdots & \vdots & \vdots & & \vdots & \vdots\\
F_{1n} & F_{2n} & F_{3n} & \cdots& F_{n-1 \, n} & F_{nn}
\end{array} \right). 
\end{equation}
\end{definition}

Given $1 \leq m < n$, consider the categories 
$$
\cat{A}^{\leq m}
= \cat{A}_1 \times \cdots \times \cat{A}_m,\quad   \cat{A}^{> m} =
\cat{A}_{m+1} \times \cdots \times \cat{A}_n.$$
All these categories can be considered, in an obvious way, as full subcategories  of $\cat{A}$. For instance, an object $A = (A_1, \dots, A_m)$ of $\cat{A}^{\leq m}$ is identified with the object  $(A_1, \dots,
A_m, 0, \dots, 0)$ of $\cat{A}$.  It is clear that there are
canonical functors (the projection functors)
$$\pi_{\leq m}: \cat{A} \longrightarrow  \cat{A}^{\leq m},\qquad \pi_{> m }: 
\cat{A} \longrightarrow \cat{A}^{> m}.$$

By Proposition \ref{comonadasmatrices}, the triangular matrix comonad 
$(F,\delta,\varepsilon)$ over $\cat{A}$,  gives rise to the triangular matrix comonads  
$(F^{\leq m},\delta_{\leq m},\varepsilon_{\leq m})$ and
$(F^{>m},\delta_{>m},\varepsilon_{>m})$ over $\cat{A}^{\leq m}$ and
$\cat{A}^{>m}$ defined, respectively,  by 
$$
F^{\leq m} = \left( \begin{array}{ccccc} F_{11} & 0 & 0 &\cdots & 0 \\
F_{12} & F_{22} & 0 & \cdots & 0 \\
\vdots & \vdots & \vdots & & \vdots \\
F_{1m} & F_{2m} & F_{3m} & \cdots & F_{mm}
\end{array} \right) 
$$
\begin{equation}\label{delta-1}\delta_{\leq m} = \left( \begin{array}{ccccc} \delta^{11} & 0 & 0 &\cdots & 0 \\
\delta^{12} & \delta^{22} & 0 & \cdots & 0 \\
\vdots & \vdots & \vdots & & \vdots \\
\delta^{1m} & \delta^{2m} & \delta^{3m} & \cdots & \delta^{mm}
\end{array} \right), \qquad \varepsilon_{\leq m} = \left( \begin{array}{cccc} \varepsilon^{1} & 0  &\cdots & 0 \\
0 & \varepsilon^{2} &  \cdots & 0 \\
\vdots & \vdots & & \vdots \\
0 & 0 & \cdots & \varepsilon^{m}
\end{array} \right) 
\end{equation}
$$
F^{>m} = \left( \begin{array}{ccccc} F_{m+1 \; m+1} & 0 & 0 &\cdots & 0 \\
F_{m+1\;m+2} & F_{m+2 \; m+2} & 0 & \cdots & 0 \\
\vdots & \vdots & \vdots & & \vdots \\
F_{m+1\;n} & F_{m+2\;n} & F_{m+3\;n} & \cdots & F_{nn}
\end{array} \right)  
$$
\begin{equation}\label{delta-2}
\delta_{> m} = \left(\begin{array}{ccccc} \delta^{m+1 \; m+1} & 0 & 0 &\cdots & 0 \\
\delta^{m+1\;m+2} & \delta^{m+2 \; m+2} & 0 & \cdots & 0 \\
\vdots & \vdots & \vdots & & \vdots \\
\delta^{m+1\;n} & \delta^{m+2\;n} & \delta^{m+3\;n} & \cdots & \delta^{n   n}
\end{array}\right), \varepsilon_{> m} = \left( \begin{array}{cccc} \varepsilon^{m+1} & 0  &\cdots & 0 \\
0 & \varepsilon^{m+2} &  \cdots & 0 \\
\vdots & \vdots & & \vdots \\
0 & 0 & \cdots & \varepsilon^{n}
\end{array} \right) 
\end{equation}

The functor $F^m : \cat{A}^{\leq m} \to \cat{A}^{>m}$ given by
\begin{equation}\label{Eq:Fm}
F^m = \left( \begin{array}{ccccc} F_{1 \; m+1} & F_{2 \; m+1}  &\cdots & F_{m \; m+1} \\
F_{1\;m+2} & F_{2 \; m+2} & \cdots & F_{m \; m+2} \\
\vdots & \vdots &  & \vdots \\
F_{1\;n} & F_{2\;n} & \cdots & F_{m n}
\end{array} \right)  
\end{equation}
is an $F^{>m}-F^{\leq m}$--bicomodule functor with the structure natural maps $\lambda_m : F^{m} \to F^mF^{\leq m}$ and $\rho_m : F^{m} \to F^{>m}F^m$ defined by
\[
\lambda_m^{ij} : F_{ij} \to \bigoplus_{k=1}^{m} F_{kj}F_{ik}, \quad \lambda_m^{ij} = \delta^{i1j} \dotplus \cdots \dotplus \delta^{imj}
\]
and
\[
\rho_m^{ij} : F_{ij} \to \bigoplus_{k=m+1}^{n} F_{kj}F_{ik}, \quad \lambda_m^{ij} = \delta^{im+1j} \dotplus \cdots \dotplus \delta^{inj},
\]
for $1 \leq i \leq m, m+1 \leq j \leq n$. Therefore, 
\[
F = \upmatrix{F^{\leq m}}{F^m}{F^{> m}},
\]
as comonads. 

The following consequence of Theorem \ref{thm:basico} is the basic tool for the study of triangular hereditary comonads in the next section. The  particular case where $F_{ii} = 1_{\mathcal{A}_i}$ for all $i =1, \dots, n$ was already stated in its dual form in \cite[Theorem 2.3]{Harada:1967} under the additional hypothesis that the categories $\mathcal{A}_i$ are abelian.

\begin{theorem}\label{un medio}
Let $F$ be a triangular matrix comonad on $\cat{A}=\cat{A}_1\times \cdots \times \cat{A}_n$ as in Definition \ref{def:comtriangular}. Given an integer $m$ with $1 \leq m < n$, assume that all equalizers do exist in $\cat{A}_i$ for $i=1, \dots, m$, and that the functors $F_{ij}$ preserve equalizers for $1 \leq i \leq j \leq m$.
Then there exists a functor 
\begin{equation}\label{eq:Tm}\xymatrix{ T^{m}: \cat{A}^{\leq m}_{F^{\leq m}} \ar@{>}[rr] & &   \cat{A}^{> m}_{F^{> m}} }
\end{equation}
 such that the category of $F$-comodules is isomorphic to the category of 
$G^m$-comodules, where $G^m$ is the normal triangular matrix comonad defined by $T^m$. That is,
$$  \cat{A}_F \,\cong \, (\,\cat{A}^{\leq m}_{F^{\leq m}} \times \cat{A}^{> m}_{F^{> m}}\,)_{G^m} ,\,\,  \text{ where  } G^m\,\,=\,\,\upmatrix{1}{T^m}{1}.$$
\end{theorem}
\begin{proof}
Apply Theorem \ref{thm:basico} with $R = F^{\leq m}$, $S = F^{>m}$, and $V = F^m$. 
\end{proof}
\begin{remark}\label{cat-T}
If the categories $\cat{A}_i$ are abelian for $i = 1, \dots, n$, and all the functors $F_{ij}$, for $1 \leq i < j \leq n$ are left exact, then the functor $F$ is also left exact, as well as the functors $F^{ > m}$, $F^m$,  and $F^{\leq m}$, for every $2 \leq m \leq n-1$.  On the other hand, the functor $T^{m}: \cat{A}^{\leq m}_{F^{\leq m}} \to \cat{A}^{> m}_{F^{> m}}$, constructed as an equalizer (see the proof of Theorem \ref{thm:basico}) 
becomes left exact, see Remark \ref{rem:leftexact}.
\end{remark}

\section{Hereditary categories of comodules.}\label{sec:III}

An abelian  category with enough injectives is said to be 
\emph{hereditary} if its global homological dimension is $0$ or $1$, that is, for every epimorphism $E_0 \to E_1$, if $E_0$ is injective, then $E_1$ is injective. Our aim is to characterize when the category of comodules $\cat{A}_F$ of a normal triangular matrix comonad $F = (F,\delta,
\varepsilon)$ over  $\cat{A} = \cat{A}_1 \times \cdots \times
\cat{A}_n$ is hereditary. We are denoting by $\delta : F \to F^2$ the comultiplication of $F$, and by $\varepsilon:F \to \cat{A}$ its counit. The shape of the matrix of functors representing $F$ is
\begin{equation}\label{nntriangular}
F = \left( \begin{array}{cccccc} 1 & 0 & 0 &\cdots & 0&0 \\
F_{12} & 1 & 0 & \cdots & 0& 0 \\
\vdots & \vdots & \vdots & & \vdots & \vdots\\
F_{1n} & F_{2n} & F_{3n} & \cdots& F_{n-1 \, n} & 1
\end{array} \right). 
\end{equation}
We assume that the categories $\cat{A}_1,
\dots, \cat{A}_n$ are abelian with enough injectives, and so is  $\cat{A} = \cat{A}_1 \times \cdots \times \cat{A}_n$. If $F : \cat{A} \rightarrow
\cat{A}$ is a left exact functor, then $\cat{A}_F$ is also abelian and the forgetful functor  $\cat{A}_F \rightarrow \cat{A}$ is exact, see \cite{Eilenberg/Moore:1965}. On the other hand, $F$ is exact if and only if  $F_{ij}$ is exact for all $i, j = 1, \dots, n$. 

 First we analyze the case $n = 2$. 

\subsection*{The case $\boldsymbol{n=2}$.}
Let $\cat{A} = \cat{A}_1 \times \cat{A}_2$, for $\cat{A}_1, \cat{A}_2$ abelian  categories, and let 
\[
F = \left( \begin{array}{cc} 1 & 0 \\ F_{12} & 1
\end{array}\right) : \cat{A}_1 \times \cat{A}_2 \longrightarrow
\cat{A}_1 \times \cat{A}_2
\]
be a functor with $F_{12} : \cat{A}_1 \rightarrow \cat{A}_2$ a left exact functor. Consider the unique structure of normal triangular comonad on the left exact functor $F$. According to Proposition \ref{catcomodulos}, any $F$--comodule can be identified with a pair $(X,d_X^{12})$, where $X = (X_1, X_2) \in \cat{A}_1 \times \cat{A}_2$ and $d_X^{12} : X_2 \to F_{12}X_1$.  By convenience, the free functor will be denoted by $F: \cat{A} \to \cat{A}_F$. By $\mathrm{Inj.dim}(X)$ we denote the injective dimension of an object $X$ in some abelian category with enough injectives. 

\begin{proposition}\label{prop:2-equiv}
Assume that  $\cat{A}_1$ and $\cat{A}_2$ have enough injectives and that  $F_{12}$ is a left exact functor that preserves injectives. Then
\begin{enumerate}
\item\label{sufiny} $\cat{A}_F$ has enough injectives. 
\item\label{estriny} Each injective comodule in $\cat{A}_F$ is, up to isomorphisms, of the form
$F (E)$, for some injective object $E$ in  $\cat{A}$. In particular, the arrow $d^{12}_X$ is a split epimorphism of $\cat{A}_2$ for any injective $F$-comodule  $(X,d^{12}_X)$. 
\item\label{desiny} Given an  $F$-comodule  $(X,d^{12}_X)$, we have 
\[
\mathrm{Inj.dim}((X,d^{12}_X)) \leq max\{
\mathrm{Inj.dim}(X_1), \mathrm{Inj.dim(X_2)}\} + 1 
\]
\end{enumerate}
\end{proposition}
\begin{proof}
(\ref{sufiny}). Observe that $F : \cat{A} \rightarrow \cat{A}_F$ preserves injectives, since it is right adjoint to the forgetful functor $\cat{A}_F \rightarrow \cat{A}$, which is exact \cite[Proposition 5.3]{Eilenberg/Moore:1965}.  So, given an injective object  $E =
(E_1,E_2)$ in $\cat{A}$, we have an injective $F$-comodule $F(E)
= (E_1, F_{12}E_1 \oplus E_2, \pi)$, where $\pi : F_{12}E_1 \oplus
E_2 \rightarrow F_{12}E_1$ is the canonical projection (this is given by the comonad structure of $F$). Now, for every 
$F$-comodule $(X,d^{12}_X)$, we can consider  monomorphisms  $\iota_i : X_i
\rightarrow E_i$ in $\cat{A}_i$, with $E_i$ injective for $i = 1,
2$. So we have a monomorphism  of $F$-comodules
\begin{equation}\label{XenFE}
(\iota_1,F_{12}\iota_1 \circ d_X^{12} \dotplus \iota_2)
:(X_1,X_2,d_X^{12}) \rightarrow (E_1,F_{12}E_1 \oplus E_2, \pi)
\end{equation}
which shows that $\cat{A}_F$ has enough injectives. 

(\ref{estriny}). If we assume that the  $F$-comodule 
$(X,d^{12}_X)$ is injective, then the monomorphism \eqref{XenFE} 
splits, so that there exists a morphism of comodules $(\alpha,\beta) :
(E_1,F_{12}E_1 \oplus E_2) \rightarrow (X_1,X_2,d_X^{12})$ which splits $(\iota_1,F_{12}\iota_1 \circ d_X^{12} \dotplus
\iota_2)$.  Thus $X_1$ is isomorphic to a direct summand of the injective object
$E_1$, so it is injective. Therefore, without loss of generality, we can suppose that  $X_1$ is injective and that $E_1
= X_1$.  In this way, we get that $\beta \circ (d_X^{12} \dotplus
\iota_2) = 1_{X_2}$, which shows that $X_2$ is isomorphic to a direct summand of the injective object $F_{12}E_1 \oplus E_2$ and so
it is injective too. On the other hand, since $(\alpha, \beta)$
is a comodule map, we obtain that  $\pi = d_X^{12} \circ \beta$.
Since $\pi$ is a split epimorphism, we deduce  that $d_X^{12}$ is a split epimorphism.  This implies that there exists an isomorphism $\omega : X_2 \rightarrow
F_{12}X_1 \oplus E_2'$ in $\cat{A}_2$ such that $\pi' \circ \omega =
d_X^{12}$, where $\pi'$ is the obvious canonical projection. Thus
\[
(1_{X_1},\omega) : (X_1,X_2,d_{X}^{12}) \rightarrow F(X_1,E_2') =
(X_1,F_{12}X_1\oplus E_2',\pi')
\]
is an isomorphism of  $F$-comodules. 

(\ref{desiny}). Put $m = max\{ \mathrm{Inj.dim}(\pi_1(X)),
\mathrm{Inj.dim(\pi_2(X))}\}$, and take,  for an $F$-comodule 
$(X,d^{12}_X)$, a resolution in $\cat{A}_F$
\[
\xymatrix{ 0 \ar[r] & (X,d^{12}_X) \ar^-{f^0}[r] &
(E^0,d^{12}_{E^0}) \ar[r] & \dots  \ar[r] &
(E^m,d^{12}_{E^m}) \ar^-{f^m}[r] & (C,d^{12}_C) \ar[r] &
0}
\]
with $(E^i,d^{12}_{E^i})$ injective for all $0 \leq i \leq m$.
We need to show that $(C, d^{12}_C)$ is injective too.
If we apply the projection functor $\pi_i$, for $i = 1,2$ to this resolution, by  part (\ref{estriny}), we get a resolution in $\cat{A}_i$ for
$\pi_i(X)$ with $\pi_i(E^k)$ injective and $d^{12}_{E^k}$
split epimorphism for $k = 0, \dots, m$. Since
$\mathrm{Inj.dim}(\pi_i(X)) \leq m$, we deduce that  $\pi_i(f^m)$ is
a split epimorphism and  $\pi_i(C)$ is injective for $i = 1, 2$.
We know that $d^{12}_C \circ \pi_2(f^m) = F_{12}(\pi_1(f^m))
\circ d^{12}_{E^m}$. Hence, $d^{12}_{C}$ is a split  epimorphism
and  $(C,d^{12}_C)$ is  injective.
\end{proof}

By $\mathrm{gl.dim}(\cat{A})$ we denote the global homological dimension of an abelian category  with enough injectives $\cat{A}$.

\begin{corollary}\label{gldim}
The assumptions are that of Proposition \ref{prop:2-equiv}. Then we have
\[
\mathrm{gl.dim}(\cat{A}_F) \leq \max \{ \mathrm{gl.dim}(\cat{A}_i)
: i =1, 2 \} +1 \leq \mathrm{gl.dim}(\cat{A}_F) + 1
\]
\end{corollary}

Next we give the desired characterization for hereditary categories of comodules over a triangular normal $(2\times 2)$-matrix comonad.

\begin{theorem}\label{thm:2-cara}
Let $\cat{A} = \cat{A}_1 \times \cat{A}_2$ for $\cat{A}_1,
\cat{A}_2$ abelian categories with enough injectives, and
$F_{12} : \cat{A}_1 \rightarrow \cat{A}_2$ is a left  exact  functor. Consider the unique normal comonad  structure on the endofunctor
\[
F = \left( \begin{array}{cc} 1 & 0 \\
F_{12} & 1 \end{array}\right) : \cat{A} \longrightarrow \cat{A}
\]
The category of comodules $\cat{A}_F$ is hereditary if, and only if, the following conditions are satisfied: 
\begin{enumerate}[(a)]
\item\label{presiny} $F_{12}$ preserves injectives. 
\item\label{basehereditarias} $\cat{A}_1$ and $\cat{A}_2$ are hereditary. 
\item\label{F12split} $F_{12}p$ is a split epimorphism for each epimorphism $p : E_1 \rightarrow E_1'$ between injective objects in $\cat{A}_1$.
\end{enumerate}
\end{theorem}
\begin{proof}
Suppose that $\cat{A}_F$ is hereditary.  Given an injective object $E_1$ in $\cat{A}_1$, we know that $F(E_1,0) = (E_1, F_{12}E_1,1_{F_{12}E_1})$ is an injective $F$-comodule. Now, since   $(0,1_{F_{12}E_1}) : (E_1, F_{12}E_1) \rightarrow
(0,F_{12}E_1)$ is an epimorphism in the hereditary category 
$\cat{A}_F$, we have that $(0,F_{12}E_1)$ is an injective object in $\cat{A}_F$. Henceforth, it is clear that  $F_{12}E_1$ is injective in  $\cat{A}_1$, from which  \eqref{presiny} is derived. The statement \eqref{basehereditarias} follows from  Corollary \ref{gldim}. For the proof of  \eqref{F12split}, given an epimorphism $p: E_1 \to E_1'$ in $\cat{A}_1$, we get  an epimorphism $(p,1) : F(E_1,0) =
(E_1,F_{12}E_1, 1) \rightarrow (E_1',F_{12}E_1,F_{12}p)$ in $\cat{A}_F$. Since $F(E_1,0)$ is injective, then so is $(E_1',F_{12}E_1,F_{12}p)$. Therefore,  $F_{12}p$ is a split epimorphism, by the characterization of the injectives in Proposition \ref{prop:2-equiv}.

Conversely,  consider an $F$--comodule
$(X,d^{12}_X)$ and a resolution in $\cat{A}_F$
\[
\xymatrix{0 \ar[r] & (X,d^{12}_X) \ar[r] & F(E_1,E_2)
\ar^{(f_1,f_2)}[r] & (C,d^{12}_C) \ar[r] & 0}
\]
with $(E_1,E_2) \in \cat{A}$ injective. Since $F_{12}$
preserves injectives, we have that $E_1$ and $F_{12}E_1 \oplus E_2$
are injective. As $f_1$ and  $f_2$ are epimorphism, and $\cat{A}_i$ is hereditary for $i= 1, 2$, we deduce that 
$C_1$ and  $C_2$  are injective. Using the fact that  $F_{12}(f_1)$ is a split epimorphism,  we obtain from the commutative diagram 
\[
\xymatrix{F_{12}E_1 \oplus E_2 \ar_{\pi}[d] \ar^-{f_2}[r] & C_2
\ar^{d^{12}_C}[d] \\
F_{12}E_1 \ar_{F_{12}f_1}[r] & F_{12}C_1}
\]
that $d^{12}_C$ is a split epimorphism and so
$(C,d^{12}_C)$ is injective. This shows that the injective dimension of $(X,d^{12}_X)$ is less or equal than $1$,  which means that $\cat{A}_F$ is a hereditary category.
\end{proof}

\subsection*{The case $\boldsymbol{n \geq 3}$.}
Let $\cat{A}=\cat{A}_1 \times \cdots \times \cat{A}_n$  denote a product of categories and $F: \cat{A} \to \cat{A}$  a normal triangular matrix comonad.

\medskip

 Propositions \ref{prop:Injectivos-n} and \ref{prop:k} below, which will be used to deduce our main result in this section (Theorem \ref{main-result}), contain part of the dual form of \cite[Theorem 3.6]{Harada:1967}. 

\begin{proposition}\label{prop:Injectivos-n}
Assume that $\cat{A}_i$ is an abelian category with enough injectives for every $i \in \{1,\cdots,n\}$, and that each of the functors
$F_{ij}: \cat{A}_i \to \cat{A}_j$, for $1 \leq i < j \leq n$, is left exact. If the category of comodules 
$\cat{A}_{F}$ is hereditary, then we have 
\begin{enumerate}[(1)]
 \item Each injective object  $(X,{\bf d}_X)$ of $\cat{A}_F$ is, up to isomorphisms, of the form $X=F(E_1, \dots,E_n)$
 for some injective object $(E_1, \dots, E_n)\in \cat{A}$. In particular each arrow $d_X^{ij}: X_j \to F_{ij}(X_i)$ is a split epimorphism. 
\item If $i<j<k$ in $\{1, \cdots, n\}$, then for every injective object $E_i$ in $\cat{A}_i$, we have $\delta^{ijk}_{E_i}:F_{ik}E_i \to F_{kj}F_{ik}E_i$ is a split epimorphism.
\item For each $m \in \{1,\dots,n-1\}$, $T\black^mp$ is a split epimorphism, for any epimorphism $p: (E,{\bf d}_E) \to (E',{\bf d}_{E'})$ between injective $F^{\leq m}$-comodules.
\end{enumerate}
\end{proposition}
\begin{proof} 
$(1)$. Since $F$ is left exact, then the forgetful functor $\cat{A}_F \to \cat{A}$ is exact. Therefore, the free functor $F : \cat{A} \to \cat{A}_F$ preserves injectives, with implies that $F(E_1, \dots, E_n)$ is an injective $F$--comodule for every  injective  $(E_1, \dots, E_n) \in \cat{A}_1 \times \cdots \times \cat{A}_n$.   Let us prove that every injective $F$--comodule is of this form by induction on $n$. For $n = 1$, there is nothing to prove, so let $n > 1$. By Theorem \ref{un medio},  we can identify
$\cat{A}_{F}$ with the category of comodules $ (\cat{A}_1 \times  \cat{A}^{> 1}_{F^{> 1}})_{G^1}$. 
By Theorem \ref{thm:2-cara}, we have that  $\cat{A}^{> 1}_{F^{> 1}}$ is hereditary and 
that  $T^{1}$ preserves injectives. In this case, $T^1 = F^1$, see equation \eqref{Eq:Fm}. 
Therefore, using Proposition \ref{prop:2-equiv}, we have that any injective object of $\cat{A}_{F}$ is of the form
$(X,{\bf d}_X)\,=\, (E_1, F^{1}(E_1)\oplus(V,{\bf d}_V))$ for some injective objects $E_1 \in \cat{A}_1$ and 
$(V,{\bf d}_V) \in \cat{A}^{> 1}_{F^{> 1}}$. By induction hypothesis we get that $(V,{\bf d}_V)\,=\, F^{> 1}(E_2,\dots, E_n)$
for some injectives $E_j$ in $\cat{A}_j$, $j=2,\cdots,n$. Hence $(X,{\bf d}_X)=F(E_1,E_2, \cdots ,E_n)$ for some injective objects 
$E_i$ in $\cat{A}_i$, $i=1,\cdots,n$.  By Theorem \ref{thm:2-cara},
\[
(d^{12}_X, \dots, d^{1n}_X) : F^1E_1 = (F_{12}E_1, \dots, F_{1n}E_1) \to (E_2, \dots, E_n)
\]
is a split epimorphism, which implies that $d^{1j}_X$ is a split epimorphism for $j=2, \dots, n$. 
 Since by induction we know that each of the $d_V^{ij}$, $2\leq i < j \leq n$ is a split epimorphism,
we deduce that every ${d}_X^{ij}$, for every $i<j$ in $\{1,\dots,n\}$, is a split epimorphism.

$(2)$. It follows from $(1)$, since any object of the form $F(0,\cdots,0,E_i,0\cdots,0)$, for $i=1,\cdots,n-2$
is injective in $\cat{A}_{F}$ (recall that the structure morphisms are exactly $$d_{F(0,\dots,0,E_i,0\dots,0)}^{jk}
\,=\, \delta^{ijk}_{E_i},  \,\, \; i<j<k).$$

$(3)$. By Theorem   \ref{un medio}, we know that $(\cat{A}^{\leq m}_{F^{\leq m}} \times \cat{A}^{> m}_{F^{>m }} )_{G^m}$ is hereditary, where $G^{m}$ is the triangular $(2\times 2)$-matrix comonad constructed by using the functor $T^m$ from \eqref{eq:Tm}. Now we  conclude by using Remark \ref{cat-T} and Theorem \ref{thm:2-cara}.
\end{proof}

\begin{proposition}\label{prop:k}
Let $n \geq 2$ be a positive integer. Assume that $\cat{A}_i$ is an abelian category with enough injectives for all $i \in \{1,\cdots,n\}$, and that
$F_{ij}: \cat{A}_i \to \cat{A}_j$, is left exact for every $1 \leq i < j \leq n$. If the category of comodules 
$\cat{A}_{F}$ is a hereditary category, then we following statements hold. 
\begin{enumerate}[(1)]
\item Each of the functors $F_{ij}$ preserves injectives.
\item $F_{ij}p$ is a split epimorphism, for any epimorphism $p: E_i \to E_i'$ between injective objects in  $\cat{A}_i$, for every $1\leq i<j \leq n$.
\end{enumerate}
\end{proposition}
\begin{proof}
$(1)$ For $n = 2$, the claim follows from Theorem \ref{thm:2-cara}. Assume $n \geq 3$, we will proceed by induction on $n$. By Theorem \ref{un medio}, the hereditary category $\cat{A}_F$ is isomorphic $(\cat{A}_1 \times \cat{A}^{> 1}_{F^{>1}})_{F^1}$ (here, $T^1 = F^1$). From Theorem \ref{thm:2-cara} we get that $\cat{A}^{> 1}_{F^{>1}}$ is hereditary. By induction hypothesis, all the functors $F_{ij}$ with $2 \leq i < j$ preserve injectives. Given an injective object $E_1$ in $\cat{A}_1$, we know from the proof of Theorem \ref{un medio} that $F^{1}E_1=((F_{12}E_1, \dots,F_{1n}E_1), {\bf d}_{F^1E_1} ),$  is injective in $\cat{A}^{> 1}_{F^{>1}}$. On the other hand, by Proposition \ref{prop:Injectivos-n}.(1)  there exists an injective object $(E_2, \dots, E_n) \in \cat{A}_2 \times \cdots \times \cat{A}_n$ such that $F^{1}E_1 = F^{>1}(E_2, \dots, E_n)$. Therefore, since each $F_{2j}$ for $j =3, \dots n$ preserves injectives, we obtain that $F_{1k}(E_1)$ is injective for every $k = 2, \dots, n$. This completes the induction. \\
$(2)$ Use induction on $n$ and Theorems \ref{un medio}, \ref{thm:2-cara}. 
\end{proof}

The following is our main result in this section.

\begin{theorem}\label{main-result}
 Assume that  $\cat{A}_i$ is an abelian category with enough injectives for all $i \in \{1,\cdots,n\}$, and consider a normal triangular $n  \times n$-matrix comonad $F = (F_{ij}) : \cat{A} \to \cat{A}$ such that 
$F_{ij}: \cat{A}_i \to \cat{A}_j$, is left exact for every $1\leq i < j \leq n$. 
Then the category of comodules
$\cat{A}_F$  is hereditary if, and only if, the following conditions are fulfilled.
\begin{enumerate}[(a)]
\item $T^{n-1}$ and every $F_{ij}$ preserves injectives for  $1 \leq i <j \leq n$.
\item  For each $m \in \{1,\dots,n-1\}$, $T^mp$ is a split epimorphism, for every epimorphism $p: (E,{\bf d}_E) \to (E',{\bf d}_{E'})$ between injective $F^{\leq m}$-comodules.

\item For every injective object $E_i$ in $\cat{A}_i$, its image  $\delta^{ijk}_{E_i}$ for $i<j<k$  is a split epimorphism.

\item Each of the categories $\cat{A}_i$ is hereditary.
\end{enumerate}
\end{theorem}
\begin{proof}
$\Leftarrow )$. We use induction on $n$. 
For $n=2$ this implication is given by Theorem \ref{thm:2-cara}, since in this case we have $T^1 =F^1=F_{12}$. 

Let $n \geq 3$, and suppose the implication is true for any category 
of comodules over a normal triangular matrix  comonad constructed by using a linearly ordered  set of length $n-1$.  Without loss of generality we can suppose by Theorem \ref{un medio} that 
$\cat{A}_{F} \,=\, (\cat{A}^{\leq n-1}_{F^{\leq n-1}} \times \cat{A}_n)_{G^{n-1}}$, where as before $G^{n-1}$ is the triangular $2\times 2$-matrix comonad associated to the functor $T^{n-1} : \cat{A}^{\leq n-1}_{F^{\leq n-1}} \to \cat{A}_n$ given in \eqref{eq:Tm}. Since the axioms $(a)$-$(d)$ are satisfied for the triangular $(n-1) \times (n-1)$-matrix comonad $F^{\leq n-1}$, by induction hypothesis we know that its category of comodules $\cat{A}^{\leq n-1}_{F^{\leq n-1}}$ is hereditary. 
Henceforth, by Remark \ref{cat-T}, Theorem \ref{thm:2-cara} can be applied as $\cat{A}_n$ is already assumed to be hereditary. Therefore, 
$\cat{A}_F$ is hereditary since $T^{n-1}$ preserves injectives.

$\Rightarrow)$. Conditions $(b)$ and $(c)$ follow from Proposition \ref{prop:Injectivos-n}. By Proposition \ref{prop:k}.(1), we know that $F_{ij}$ preserves injectives for every $1 \leq i < j \leq n$. Since, by Theorems \ref{un medio} and \ref{thm:2-cara}, we get that $T^{n-1}$ preserves injectives, we conclude $(a)$.
We use induction to prove $(d)$. For $n=2$, this implication is clear from Theorem \ref{thm:2-cara}. Suppose that $(d)$ holds for any hereditary category of comodules over a triangular matrix comonad which was constructed by using a 
linearly ordered set of length $n-1$. By Remark \ref{cat-T}, we know that $T^{m}$ is a left exact functor 
for any $m \in \{1, \cdots,n-1\}$. Thus by Theorems \ref{thm:2-cara} and \ref{un medio}, we know that 
$\cat{A}^{\leq n-1}_{F^{\leq n-1}}$ and $\cat{A}_n$ are hereditary. Hence $\cat{A}_i$, $i=1,\cdots,n$ are hereditary. 
This gives to us condition $(d)$.  
\end{proof}

\section{Hereditary triangular matrix coalgebras.}\label{sec:coalgebras}

We illustrate our results by applying some of them to categories of comodules over coalgebras. We refer to \cite{Brzezinski/Wisbauer:2003} for basic information on coalgebras over commutative rings and their categories of comodules. 

\begin{punto}\textbf{Triangular matrix coalgebras.}
Let $C$ and $D$  be coalgebras over a commutative ring $K$, and $M$ be a $C-D$--bicomodule. Then the functors $- \tensor{K} C, - \tensor{K} D$ give comonads over the category $\rmod{K}$ of $K$--modules, and $- \tensor{K} M : \rmod{K} \to \rmod{K}$ becomes a $(- \tensor{K} D)-( - \tensor{K}C)$--bicomodule functor (see Subsection \ref{3comonad}). These functors define a triangular matrix comonad on $\rmod{K} \times \rmod{K} = \rmod{K \times K}$ represented by the $K \times K$--coalgebra  
$$ E = \scriptscriptstyle{ \upmatrix{C}{M}{D}}.$$
The comultiplication of this coalgebra over $R = K \times K$ is given by 
\begin{footnotesize}
\begin{multline}\label{Delta}
\Delta  \begin{pmatrix}
  c & 0 \\
  m & d
\end{pmatrix} = \underset{(c)}\sum  \begin{pmatrix}
  c_{(1)} & 0 \\
  0 & 0
\end{pmatrix} \tensor{R}  \begin{pmatrix}
  c_{(2)} & 0 \\
  0 & 0
\end{pmatrix} + \underset{(d)}\sum  \begin{pmatrix}
  0 & 0 \\
  0 & d_{(1)}
\end{pmatrix} \tensor{R}  \begin{pmatrix}
  0 & 0 \\
  0 & d_{(2)}
\end{pmatrix}  \\ + \underset{(m)}\sum\begin{pmatrix}
  m_{(-1)} & { 0} \\
  0 & 0
\end{pmatrix} \tensor{R}  \begin{pmatrix}
  0 &  0 \\ 
  m_{(0)} & 0
\end{pmatrix} +  \underset{(m)}\sum\begin{pmatrix}
  0 & 0 \\
 m_{(0)}   & 0
\end{pmatrix} \tensor{R}  \begin{pmatrix}
  0 & 0 \\
  0 & m_{(1)}
\end{pmatrix},
\end{multline}
\end{footnotesize}
and  its counity is defined by 
\[
\varepsilon \begin{pmatrix}
  c & 0 \\
  m & d
\end{pmatrix} = \begin{pmatrix}
  \varepsilon_{C}(c) & 0 \\
  0 & \varepsilon_{D}(d)
\end{pmatrix},
\]
where we are using Heyneman-Sweedler's notation. 

Assume that $C$ and $M$ are flat as $K$--modules. The category of right $C$--comodules is denoted by $\rcomod{C}$, and similarly for any other coalgebra over a commutative ring. By Theorem \ref{thm:basico},  we have the normal triangular matrix comonad $ G  =\scriptscriptstyle{ \upmatrix{1}{- \cotensor{C} M}{1}}$  on  $\rcomod{C} \times \rcomod{D}$, and the equivalence of categories
\begin{equation}\label{eq1}
(\rcomod{C} \times \rcomod{D})_{ G } \cong  (\rmod{K \times K})_{- \tensor{R} E} = \rcomod{E}.
\end{equation}

\end{punto}

\begin{punto}\textbf{Base change ring by a Frobenius algebra.}\label{cambiodebase}
 Let $R$ be  a commutative Frobenius algebra over a commutative ring $K$ (see \cite{Kasch:1960}), that is, the functor $- \otimes_K R : \rmod{K} \to \rmod{R}$ is right adjoint to the forgetful functor from $\rmod{R}$ to $\rmod{K}$ (see \cite{Morita:1965} and \cite[Remark 2.3.1]{Castano/alt:1999}). The counit  of this adjunction evaluated at $K$ gives the  Frobenius functional $\psi : R \to K$. If $\eta$ denotes the unit, then $\eta_R(1) = \sum_i e_i \otimes f_i \in R \tensor{K} R$ is the Casimir element.  Given any $R$--coalgebra $E$,  we get the adjoint pairs of functors 
\[
\xymatrix@C=40pt{\rcomod{E} \ar@<0.5ex>[r] & \ar@<0.5ex>^-{- \tensor{R} E }[l]\rmod{R} \ar@<0.5ex>[r] & \ar@<0.5ex>^-{- \otimes_K R} [l] \rmod{K},   }
\]
where the unlabelled arrows denote the forgetful functors that are left adjoints to $- \tensor{R} E$ and $- \otimes_K R$. 
 By composing these adjoint pairs we obtain the adjoint pair
\begin{equation}\label{adjcompuesta}
\xymatrix@C=40pt{\rcomod{E} \ar@<0.5ex>[r] &  \ar@<0.5ex>^-{- \otimes_K E }[l] \rmod{K}}, 
\end{equation}
where we are using the isomorphism $E \cong R \tensor{R} E$. In this way, we obtain a comonad $- \tensor{K} E : \rmod{K} \to \rmod{K}$, which is determined by the structure of $K$--coalgebra on $E$ with comultiplication 
\[
\xymatrix{E \ar^-{\tilde{\Delta}}[r] & E \otimes_K E, &  x \ar@{|->}[r] & \sum_{(x),i}x_{(1)}e_i \otimes f_i x_{(2)}  }
\]
and counity $\tilde{\varepsilon} = \psi \circ \varepsilon$, where $\varepsilon : E \to R$ the counity of the $R$--coalgebra $E$. 

The adjoint pair \eqref{adjcompuesta} gives rise to the comparison functor $V : \rcomod{E} \to (\rmod{K})_{- \otimes E}$ (see \cite[Section 3.2]{Barr/Wells:1985}). If $E$ is flat as an $R$--module, then $\rcomod{E}$ is an abelian category and the faithful forgetful functor $\rcomod{E} \to \rmod{R}$ is exact \cite{Eilenberg/Moore:1965}. This easily implies that the forgetful functor $\rcomod{E} \to \rmod{K}$ satisfies the hypotheses of Beck's theorem (precisely, the dual of \cite[Theorem 3.3.10]{Barr/Wells:1985}) and, hence, $V$ is an equivalence of categories. We get thus that the categories of right comodules over the $R$--coalgebra $(E,\Delta,\varepsilon)$ and of right comodules over the $K$--coalgebra $(E,\tilde{\Delta},\tilde{\varepsilon})$ are equivalent. \end{punto}

\begin{punto}\textbf{Bipartite coalgebras.}
Let $C, D$ be $K$--coalgebras, and $M$ be a $C-D$--bicomodule. Assume that $C, D, M$ are flat $K$--modules. 
In the case of the coalgebra $E = \upmatrix{C}{M}{D}$, the base ring $R = K \times K$ is a Frobenius $K$--algebra with Frobenius functional $\psi : R \to K$ given by $\psi (a,b) = a + b$ for all $(a,b) \in K \times K$, and Casimir element $e = u \otimes u + v \otimes v \in R \otimes R$, where $u = (1, 0), v  = (0,1)$. By  \ref{cambiodebase} ,  $E$ is a  $K$--coalgebra.  Explicitly, the comultiplication and counity are

\begin{footnotesize}
\begin{multline}\label{coalDelta}
\tilde{\Delta}  \begin{pmatrix}
  c & 0 \\
  m & d
\end{pmatrix} = \underset{(c)}\sum  \begin{pmatrix}
  c_{(1)} & 0 \\
  0 & 0
\end{pmatrix} \tensor{}  \begin{pmatrix}
  c_{(2)} & 0 \\
  0 & 0
\end{pmatrix} + \underset{(d)}\sum  \begin{pmatrix}
  0 & 0 \\
  0 & d_{(1)}
\end{pmatrix} \tensor{}  \begin{pmatrix}
  0 & 0 \\
  0 & d_{(2)}
\end{pmatrix}  \\ + \underset{(m)}\sum\begin{pmatrix}
  m_{(-1)} & { 0} \\
  0 & 0
\end{pmatrix} \tensor{}  \begin{pmatrix}
  0 &  0 \\ 
  m_{(0)} & 0
\end{pmatrix} +  \underset{(m)}\sum\begin{pmatrix}
  0 & 0 \\
 m_{(0)}   & 0
\end{pmatrix} \tensor{}  \begin{pmatrix}
  0 & 0 \\
  0 & m_{(1)}
\end{pmatrix},
\end{multline}
\end{footnotesize}
and
\[
\tilde{\varepsilon} \begin{pmatrix}
  c & m \\
  0 & d
\end{pmatrix} = \varepsilon_C(c)  + \varepsilon_D (d)
\] 
We recover thus the construction of a bipartite $K$--coalgebra from \cite[p. 91]{Kosakowska/Simson:2008} (called triangular matrix coalgebra in  \cite{Iovanov:2015}).
\end{punto}

In the following theorem we say that a $K$--coalgebra is \emph{right hereditary} if  the category $\rcomod{C}$ is hereditary. 

\begin{theorem}\label{3hered}
Let $C$, $D$ be $K$--coalgebras, and $M$ be a $C$--$D$--bicomodule. Assume that $C, D$ and $M$ are flat as $K$--modules. The bipartite $K$--coalgebra $\upmatrix{C}{M}{D}$ is right hereditary if and only if the following conditions hold.
\begin{enumerate}
\item $ U \cotensor{C} M $ is an injective right $D$--comodule for every injective right $C$--comodule $U$;
\item $C$ and $D$ are right hereditary;
\item $p \cotensor{C} M$ is a split epimorphism for each epimorphism $p : E_1 \to E_1'$ of injective right $C$--comodules $E_1, E_1'$. 
\end{enumerate}
\end{theorem}
\begin{proof}
Since we assume $M$, $C$ and $D$ to be flat over $K$, it follows that the functor $- \cotensor{C} M : \rcomod{C} \to \rcomod{D}$ is left exact. Now, the theorem follows from Theorem \ref{thm:2-cara}, the equivalence of categories \eqref{eq1}, and the equivalence of categories given at the end of paragraph \ref{cambiodebase}. 
\end{proof}

\begin{punto}\textbf{Generalized matrix coalgebras.}
Let $M_{ij}$,  $i,j = 1, \dots, n$,  be a set of modules over a commutative ring $K$, and consider the endofunctors $F_{ij} =  M_{ij} \otimes_{K} - : \rmod{K} \to \rmod{K}$. Then the matrix  functor $F : \rmod{K}^n \to \rmod{K}^n$ has the structure of a comonad if and only if there exists a set of natural transformations $\delta^{ikj}$ and $\varepsilon^i$ as in Proposition \ref{comonadasmatrices}. These natural transformations are determined by linear maps $\phi_{ikj} = \delta^{ikj}_K : M_{ij} \to M_{kj} \otimes M_{ik}$ and $\epsilon_i = \varepsilon^i_K : M_{ii} \to K$.  We thus obtain a $K^n$-coalgebra 
\[
F(K) = \left( \begin{array}{cccc}
M_{11} & M_{21} & \cdots & M_{n1} \\
M_{12} & M_{22} & \cdots & M_{n2} \\
\vdots & \vdots & & \vdots \\
M_{1n} & M_{2n} & \cdots & M_{nn} 
\end{array}\right).
\]
By Remark \ref{remark:CT}, each of the entries in the diagonal of this matrix is a $K$--coalgebra, and $M_{ij}$ is an $M_{jj}-M_{ii}$--bicomodule. Also, every $\phi_{ikj}$  factors trough the cotensor product $M_{kj} \square_{M_{kk}} M_{ik}$. Using the base change of ring from \ref{cambiodebase} with the Frobenius $K$--algebra $R = K^n$ we get the `comatrix' $K$-coalgebra of \cite[Section 2]{Iovanov:2015}. We prefer the name matrix coalgebras to avoid confusion with the notion of a comatrix coring (and, in particular, comatrix coalgebra) from \cite{Kaoutit/Gomez:2003}, which is a different construction. As in \ref{cambiodebase}, the comultiplication and counit of this matrix $K$--coalgebra can be computed explicitly. 
\end{punto}

\begin{punto}\textbf{Triangular matrix corings.}
More generally, we may consider corings $\coring{C}$ and $\coring{D}$ over different base rings $A$ and $B$, respectively (see \cite{Brzezinski/Wisbauer:2003}). Given a $\coring{C}-\coring{D}$--bicomodule $M$, we get the triangular matrix $A \times B$-coring $\upmatrix{\coring{C}}{M}{\coring{D}}$. Under suitable flatness conditions, a  result  similar to Theorem \ref{3hered} may be formulated for corings.

\begin{remark}\label{remark:coanillo}
Let $\coring{C}$ be an $R$-coring which is flat as left $R$-module.  Assume that $R$ is a (possibly non commutative) Frobenius algebra over a commutative ring $K$.  Since the functor $R\tensor{K}-: \rmod{K} \to \rmod{R}$ is then a right adjoint to the forgetful functor from the category of right $R$--modules $\rmod{R}$ to $\rmod{K}$, the arguments from \ref{cambiodebase} run here to prove that the category of right  $\coring{C}$--comodules is equivalent to the category of right comodules over a certain $K$-coalgebra built from $R$ and $\coring{C}$. In particular, when $\coring{C}=R$ endowed with the trivial coring structure, we get the main result from \cite{Abrams:1999}, namely, that the category of right $R$-modules is equivalent to the category of right $R$-comodules.
\end{remark}
\end{punto}

\providecommand{\bysame}{\leavevmode\hbox to3em{\hrulefill}\thinspace}
\providecommand{\MR}{\relax\ifhmode\unskip\space\fi MR }
% \MRhref is called by the amsart/book/proc definition of \MR.
\providecommand{\MRhref}[2]{%
  \href{http://www.ams.org/mathscinet-getitem?mr=#1}{#2}
}
\providecommand{\href}[2]{#2}

\end{document}